\documentclass{amsart}
\usepackage[utf8]{inputenc}
\usepackage[english]{babel}
\usepackage{amsthm}
\usepackage{amssymb}
\usepackage{amsmath}
\usepackage{graphicx} 
\usepackage{tikz}
\usepackage{soul}
\usepackage{ytableau,tikz,varwidth}
\usetikzlibrary{calc}
\usepackage{tikz-cd}

\tikzset{
  curarrow/.style={
  rounded corners=8pt,
  execute at begin to={every node/.style={fill=red}},
    to path={-- ([xshift=-50pt]\tikztostart.center)
    |- (#1) node[fill=white] {$\scriptstyle \delta$}
    -| ([xshift=50pt]\tikztotarget.center)
    -- (\tikztotarget)}\
    }
}

\newtheorem{theorem}{Theorem}[section]
\newtheorem{corollary}[theorem]{Corollary}
\newtheorem{lemma}[theorem]{Lemma}
\newtheorem{prop}[theorem]{Proposition}
\theoremstyle{definition}

\newtheorem{remark}[theorem]{Remark}

\newcommand{\kk}{\Bbbk}

\newcommand{\calF}{\mathcal{F}}
\newcommand{\calG}{\mathcal{G}}

\newcommand{\AAA}{\mathbb{A}}

\newcommand{\ZZ}{\mathbb{Z}}
\newcommand{\QQ}{\mathbb{Q}}
\newcommand{\RR}{\mathbb{R}}

\newcommand{\eb}{\mathbf{e}}

\def\opn#1#2{\def#1{\operatorname{#2}}} 
\opn\Cl{Cl} \opn\pdim{pdim} \opn\Im{Im} \opn\Ker{Ker} \opn\ini{in}

\title{The $h$-vectors of the edge rings of a special family of graphs}
\author{Akihiro Higashitani}
\address{Department of Pure and Applied Mathematics, Graduate School of Information Science and Technology, Osaka University, Suita, Osaka 565-0871, Japan}
\email{higashitani@ist.osaka-u.ac.jp}
\author{Nayana Shibu Deepthi}
\address{Department of Pure and Applied Mathematics, Graduate School of Information Science and Technology, Osaka University, Suita, Osaka 565-0871, Japan}
\email{nayanasd@ist.osaka-u.ac.jp}

\begin{document}

\begin{abstract}
The $h$-vectors of homogeneous rings are one of the most important invariants that often reflect ring-theoretic properties. 
On the other hand, there are few examples of edge rings of graphs whose $h$-vectors are explicitly computed. 
In this paper, we compute the $h$-vector of a special family of graphs, by using the technique of initial ideals and the associated simplicial complex. 
\end{abstract}

\maketitle

\section{Introduction}

Throughout the paper, we assume all graphs to be connected and having no loops and multiple edges.

Many researchers have conducted intensive studies on the edge rings and toric ideals of graphs. 
Especially, Ohsugi and Hibi initiated the investigation and have been developing the theory of edge rings. 
For instance, the characterization for the edge rings to be normal was given in \cite{OH98}. 
Note that, almost at the same time, the same result was obtained by Simis-Vasconcelos-Villarreal independently in \cite{SVV}. 
For the introduction to the edge rings and toric ideals of graphs, we refer the readers to \cite[Section  5]{HHO} and \cite[Section 10]{Villa}. 

One of the most important invariants of homogeneous rings is the Hilbert series. 
In fact, the Gorensteinness of homogeneous normal Cohen--Macaulay domains is characterized by the symmetry of their $h$-vectors (\cite{S78}). 
Moreover, there are many other results claiming that the $h$-vectors of homogeneous (or semi-standard graded) normal Cohen--Macaulay rings (or domains) 
have some connection with their ring-theoretic properties (see, e.g., \cite{BD, H, HY, Y}, and so on). 
On the other hand, the $h$-vector of an edge ring is scarcely known. 
As far as we know, the $h$-vectors (or their counterparts) of the edge rings of the following graphs have been computed: 
\begin{itemize}
\item Let $K_{m,n}$ denote the complete bipartite graph with $m+n$ vertices. Then 
$$h(\kk[K_{m,n}];t)=\sum_{i=0}^{\min\{m,n\}}\binom{m-1}{i}\binom{n-1}{i}t^i.$$
\item Let $K_m$ denote the complete graph with $m$ vertices. Then 
$$h(\kk[K_m];t)=1+\frac{m(m-3)}{2}t+\sum_{i=2}^{\left\lfloor \frac{m}{2} \right\rfloor}\binom{m}{2i}t^i.$$
\item The Hilbert functions of the edge rings of complete multipartite graphs were computed in \cite[Theorem 2.6]{OH00}. 
\item The Hilbert series of the edge rings of bipartite graphs are described using their interior polynomials (\cite{KP}). 
\end{itemize}
(On the notations used above, see Section~\ref{sec:pre}.) 
Regarding the results on $K_{m,n}$ and $K_m$, see \cite{V96}, or \cite[Section 10]{Villa}. 
Note that if the edge rings of the graphs are normal, then the Hilbert functions (resp. $h$-vectors) of the edge rings 
agree with the ``Ehrhart polynomials (resp. $h^*$-vectors) of the edge polytopes arising from the graphs''. 
The goal of this paper is to explicitly compute the $h$-vector for some particular class of graphs in order to suggest more examples. 

Let $n \geq 2$ be an integer. We introduce a connected non-bipartite graph $\mathcal{G}_{n}$ as shown in Figure~\ref{fig1}. 
Clearly, $\mathcal{G}_{n}$ satisfies the odd cycle condition (see Section~\ref{sec:pre}). 

\begin{figure}[h]
\centering
\begin{tikzpicture}
\draw[black, thin] (3,1) -- (5,2) -- (2.6,2.5)-- cycle;
\draw[black, thin] (5,2) -- (4,4) -- (6,4)-- cycle;
\draw[black, thin] (5,2) -- (7.4,2.5) -- (7.1,1)-- cycle;
\filldraw [black] (3,1) circle (1pt);
\filldraw [black] (5,2) circle (1pt);
\filldraw [black] (2.6,2.5) circle (1pt);
\filldraw [black] (4,4) circle (1pt);
\filldraw [black] (6,4) circle (1pt);
\filldraw [black] (7.4,2.5) circle (1pt);
\filldraw [black] (7.1,1) circle (1pt);
\filldraw [black] (6.5,2.9) circle (0.6pt);
\filldraw [black] (6.3,3.3) circle (0.6pt);
\filldraw [black] (6.4,3.1) circle (0.6pt);
\filldraw [black] (7.5,2) node[anchor=north] {$z_{n}$};
\filldraw [black] (7.2,2.8) node[anchor=west] {$u^{(1)}_{n}$};
\filldraw [black] (6.1,2.5) node[anchor=west] {$x_{n}$};
\filldraw [black] (6.4,1.2) node[anchor=east] {$y_{n}$};
\filldraw [black] (4.8,1.8) node[anchor=west] {$w$};
\filldraw [black] (7.1,1) node[anchor=west] {$u^{(2)}_{n}$};
\filldraw [black] (2.3,2.8) node[anchor=west] {$u^{(2)}_{1}$};
\filldraw [black] (2.8,0.8) node[anchor=west] {$u^{(1)}_{1}$};
\filldraw [black] (4.2,2.9) node[anchor=south] {$x_{2}$};
\filldraw [black] (3.4,4.2) node[anchor=west] {$u^{(1)}_{2}$};
\filldraw [black] (6.1,3) node[anchor=east] {$y_{2}$};
\filldraw [black] (5.9,4.2) node[anchor=west] {$u^{(2)}_{2}$};
\filldraw [black] (5.3,4.2) node[anchor=east] {$z_{2}$};
\filldraw [black] (3.8,2.2) node[anchor=south] {$y_{1}$};
\filldraw [black] (2.9,1.7) node[anchor=east] {$z_{1}$};
\filldraw [black] (4.4,1.3) node[anchor=east] {$x_{1}$};
\end{tikzpicture}
\caption{The graph $\calG_n$}
\label{fig1}
\end{figure}

Our main result of this paper is the following: 
\begin{theorem}\label{thm:main}
The $h$-polynomial of $\kk[\calG_n]$ is as follows:
$$h(\kk[\calG_n];t)=\binom{n}{0}+\left(\binom{n}{1}-1\right)t+\binom{n}{2}t^2+\cdots+\binom{n}{n}t^n=(1+t)^n - t.$$
Moreover, $\kk[\calG_n]$ is almost Gorenstein but not Gorenstein if $n \geq 3$. 
\end{theorem}

The notion of almost Gorenstein homogeneous rings was introduced by Goto-Takahashi-Taniguchi in \cite{GTT}, 
as a new class of graded rings between Cohen--Macaulay rings and Gorenstein rings. 
After this work, almost Gorenstein homogeneous rings have been studied further, e.g., in \cite{H, MM}. 
On  almost Gorensteinness of edge rings, known examples of almost Gorenstein non-Gorenstein edge rings are presumably rare. 
However, almost Gorenstein edge rings arising from complete multipartite graphs were completely characterized in \cite{HM2}. 
According to \cite[Examples 1.5-1.7]{HM2}, we know that $K_{2,m}$, $K_{1,1,m}$, $K_{1,m,m}$ with $m \geq 3$ and $K_{1,1,m,m}$ with $m \geq 2$ 
give almost Gorenstein but not Gorenstein edge rings, where $K_{a_1,\ldots,a_r}$ denotes the complete $r$-partite graph.  
Theorem~\ref{thm:main} gives a new family of graphs whose edge rings are almost Gorenstein but not Gorenstein. 

\bigskip

A brief structure of this paper is as follows. 
In Section~\ref{sec:pre}, we recall some fundamental notions on homogeneous rings, 
and prepare some materials on the edge rings and toric ideals of graphs for the computation of the $h$-polynomial. 
In Section~\ref{sec:h-poly}, we compute the $h$-polynomial of $\kk[\calG_n]$ by using a certain initial complex of the toric ideal of $\calG_n$. 
In Section~\ref{sec:almGor}, we prove the almost Gorensteinness of $\kk[\calG_n]$. 
On the non-Gorensteinness of $\kk[\calG_n]$ for $n \geq 3$, see Remark~\ref{rem:Gor}.

\section{Preliminaries}\label{sec:pre}

First, we recall some fundamental materials like, notions on commutative algebra, the edge rings and toric ideals of graphs. 

\subsection{Homogeneous rings}
In this subsection, we collect some notions on homogeneous rings used in this paper. 
We refer the readers to \cite{BH, Villa} for more detailed information on homogeneous rings. 

Let $R$ be a Cohen--Macaulay homogeneous ring of dimension $d$ over a field $\kk$. 
Then the Hilbert series of $R$ looks as follows:  
$$\sum_{i \geq 0}\dim_\kk R_i t^i=\frac{h_0+h_1t+\cdots+h_st^s}{(1-t)^d},$$
where $R_i$ denotes the homogeneous part of $R$ of degree $i$, $\dim_\kk$ denotes the dimension as a $\kk$-vector space, and we assume that $h_s \neq 0$. 
We call the polynomial $h_0+h_1t+\cdots+h_st^s$ appearing in the numerator as the \textit{$h$-polynomial} of $R$, denoted by $h(R;t)$, 
and the sequence of the coefficients $(h_0,h_1,\ldots,h_s)$ as the \textit{$h$-vector} of $R$, denoted by $h(R)$.

\begin{remark}\label{rem:Gor}
It follows from \cite{S78} that, if $R$ is a normal Cohen--Macaulay homogeneous domain, then $R$ is Gorenstein if and only if $h(R)$ is symmetric, 
i.e., $h_i=h_{s-i}$ for $i=0,1,\ldots,s$. 
Hence, once we get the $h$-vector of $\kk[\calG_n]$, as described in Theorem~\ref{thm:main}, we obtain that $\kk[\calG_n]$ is Gorenstein if and only if $n \leq 2$. 
\end{remark}

We denote the \textit{Cohen--Macaulay type} of $R$ by $r(R)$. 
It is well known that $r(R)$ coincides with the number of elements in the minimal system of generators of the canonical module of $R$.

\subsection{Edge rings}
In this subsection, we recall the definitions of edge rings and toric ideals of graphs, and some fundamental theorems regarding them. 

First of all, let us recall the definition of edge rings. 
Let $G$ be a graph on the vertex set $[d]$ with edge set $E(G)=\{e_1,\ldots,e_m\}$. 
We define a polynomial ring $R=\kk[t_1,\ldots,t_d]$ in $d$ variables and another one $S=\kk[x_1,\ldots,x_m]$ in $m$ variables, where $\kk$ is a field. 
Let $\pi : S \rightarrow R$ be the ring homomorphism defined by $\pi(x_i)={\bf t}^{e_i}$ for $i=1,\ldots,m$, 
where ${\bf t}^e:=t_{u}t_{v}$ for any edge $e=\{u,v\} \in E(G)$.
The image $\Im(\pi)$, which is a subalgebra of $R$, is called the \textit{edge ring} of $G$, 
and the kernel $\Ker(\pi)$, an ideal of $S$, is called the \textit{toric ideal} of $G$. 
We denote the edge ring of $G$ by $\kk[G]$ and the toric ideal of $G$ by $I_G$. Clearly, we have the ring isomorphism $\kk[G] \cong S/I_G$. 
It is known that $\kk[G]$ is $d$-dimensional if $G$ is not bipartite  (cf. \cite[Proposition 1.3]{OH98}). 

Given an edge $e=\{i,j\} \in E(G)$, let $\rho(e)=\eb_i+\eb_j$, where $\eb_1,\ldots,\eb_d \in \RR^d$ are the unit vectors of $\RR^d$. 
Let $A_G=\{\rho(e) : e \in E(G)\}$. For $\AAA \subset \RR^d$, let $\displaystyle \AAA A_G=\left\{\sum_{e\in E(G)}a_e\rho(e) :a_e \in \AAA\right\}$. 
In this paper, we only consider the cases where $\AAA$ is $\QQ_{\geq 0}$ or $\ZZ$ or $\ZZ_{\geq 0}$. 
We can regard $\kk[G]$ as a monoid algebra of an affine monoid $\ZZ_{\geq 0}A_G$. 
In particular, the structure of the cone $\QQ_{\geq 0}A_G$ plays a crucial role in the study of $\kk[G]$. 

\medskip

Next, let us describe the cone $\QQ_{\geq 0}A_G$ in terms of linear inequalities. 
For the description, we need to introduce some more notions on graphs. 
\begin{itemize}
\item For a subset $W \subset V(G)$, let $G \setminus W$ be the subgraph on $V(G) \setminus W$ with the edge set $\{e \in E(G) : e \subset V(G) \setminus W\}$. 
If $W=\{w\}$, then we write $G \setminus w$ instead of $G \setminus \{w\}$. 
\item For $v \in V(G)$, let $N_G(v)=\{u \in V(G) : \{u,v\} \in E(G)\}$, and for any subset $W \subset V(G)$, let $N_G(W)=\bigcup\limits_{w \in W}N_G(w)$. 
\item A non-empty subset $T \subset V(G)$ is called an \textit{independent set} if $\{v,w\} \not\in E(G)$ for any $v,w \in T$. 
\item We call a vertex $v$ of $G$ \textit{regular} if each connected component of $G \setminus v$ contains an odd cycle. 
\item We say that an independent set $T$ of $V(G)$ is a \textit{fundamental set} if 
\begin{itemize}
\item the bipartite graph on the vertex set $T \cup N_G(T)$ with the edge set $\{\{v,w\} : v \in T, w \in N_G(T)\} \cap E(G)$ is connected, and  
\item $T \cup N_G(T)=V(G)$, or each of the connected components of the graph $G \setminus (T \cup N_G(T))$ contains an odd cycle. 
\end{itemize}
\end{itemize}

It follows from \cite[Theorem 1.7 (a)]{OH98} that $\QQ_{\geq 0}A_G$ consists of the elements $(x_v)_{v \in V(G)} \in \QQ^d$ satisfying all the following inequalities: 
\begin{equation}\label{eq:ineq}
\begin{split}
x_u &\geq 0 \;\;\text{ for any regular vertex }u; \\
\sum_{v \in N_G(T)}x_v &\geq \sum_{u \in T}x_u \;\;\text{ for any fundamental set }T. 
\end{split}
\end{equation}

\medskip

Next, from \cite[Section 5.3]{HHO}, we shall recall how to describe the generators of the toric ideal of a graph.
Let $G$ be a graph. Given $\Gamma=(e_{i_1},\ldots,e_{i_{2r}})$, a sequence of edges of $G$, 
we call $\Gamma$ an \textit{even closed walk} if $|e_{i_j} \cap e_{i_{j+1}}|=1$ for $j=1,\ldots,2r$, where $i_{2r+1}=i_1$. 
Moreover, we associate a binomial $f_\Gamma=f_\Gamma^{(+)}-f_\Gamma^{(-)}$, where 
$$f_\Gamma^{(+)}=\prod_{k=1}^rx_{i_{2k-1}} \;\text{ and }\;f_\Gamma^{(-)}=\prod_{k=1}^rx_{i_{2k}}.$$
We say that an even closed walk $\Gamma$ is \textit{primitive} if there is no even closed walk $\Gamma'$ of $G$ with $f_{\Gamma'} \neq f_\Gamma$ 
such that $f_{\Gamma'}^{(+)}$ divides $f_\Gamma^{(+)}$ and $f_{\Gamma'}^{(-)}$ divides $f_\Gamma^{(-)}$. 
\begin{lemma}[{\cite[Lemma 5.10]{HHO}}]
The toric ideal $I_G$ is generated by the binomials $f_\Gamma$ for all primitive even closed walks $\Gamma$. 
\end{lemma}

\medskip

Finally, we recall the condition for the normality of $\kk[G]$.
We say that $G$ satisfies the \textit{odd cycle condition} if for any pair of odd cycles $C$ and $C'$ in $G$, 
if $C$ and $C'$ do not share any common vertex, then there is a bridge between $C$ and $C'$, 
i.e., there exist $v \in V(C)$ and $v' \in V(C')$ such that $\{v,v'\} \in E(G)$. 

We see that $\kk[G]$ is normal if and only if 
\begin{align}\label{eq:normal}
\ZZ_{\geq 0}A_G = \QQ_{\geq 0}A_G \cap \ZZ A_G
\end{align}
holds. (See, e.g., \cite[Section 6.1]{BH}.) On the normality of $\kk[G]$, the following is known: 
\begin{theorem}[{\cite{OH98, SVV}}]
A graph $G$ satisfies \eqref{eq:normal} if and only if $G$ satisfies the odd cycle condition. 
\end{theorem}
Remark that $\kk[G]$ is a normal Cohen--Macaulay homogeneous domain if $G$ satisfies the odd cycle condition. 
Cohen--Macaulayness follows from Hochster's theorem (see \cite[Theorem 6.3.5]{BH}).

\subsection{Fundamental properties of $\kk[\calG_n]$}

Now, let us focus on our graph $\calG_n$. We identify the edges of $\calG_n$ with the variables of the polynomial ring $S$, as depicted in Figure~\ref{fig1}. 
Namely, we consider $$S=\kk[x_1,y_1,z_1,\ldots,x_n,y_n,z_n]$$ and we regard $I_{\calG_n}$ as an ideal of $S$. 

For $\mathcal{G}_{n}$, we see that, from \cite[Lemma 5.11]{HHO}, 
every primitive even closed walk consists of two $3$-cycles with exactly one common vertex, and is given by 
$$(x_{i},z_{i},y_{i},x_{j},z_{j},y_{j}); \hspace{0.3cm}1\leq i < j \leq n.$$
Hence, the toric ideal $I_{\mathcal{G}_{n}}$ is generated by the binomials:
\begin{align}\label{binom:G_n}
x_{i}y_{i}z_{j}-z_{i}x_{j}y_{j}; \hspace{0.3cm}1\leq i < j\leq n .
\end{align}
Let $<_\text{lex}$ be the graded lexicographic order on $S$ induced by the ordering of the variables 
\begin{align}\label{eq:lex} x_{1}<_\text{lex}y_{1}<_\text{lex}z_{1}<_\text{lex}\dots <_\text{lex} x_{n}<_\text{lex}y_{n}<_\text{lex}z_{n}. \end{align}
For the fundamental materials on initial ideals and Gr\"{o}bner basis, consult, e.g., \cite[Section 1]{HHO}. 
\begin{lemma}
The binomials in \eqref{binom:G_n} form a Gr\"{o}bner basis of $I_{\mathcal{G}_{n}}$ with respect to the monomial order $<_\text{lex}$.  
\end{lemma}
\begin{proof}
The result follows from the straightforward application of Buchberger's criterion to the set of generators \eqref{binom:G_n} of $I_{\calG_n}$. 
(For Buchberger's criterion, see, e.g., \cite[Theorem 1.29]{HHO}.) 

Let $f=x_{i}y_{i}z_{j}-z_{i}x_{j}y_{j}$ and $g=x_{p}y_{p}z_{q}-z_{p}x_{q}y_{q}$ be two generators. If $i\neq p$ and $j\neq q$, 
then the leading terms of $f$ and $g$ are relatively prime and thus the $S$-polynomial $S(f,g)$ will reduce to $0$ by \cite[Lemma 1.27]{HHO}. 

Suppose $i=p$. Then
\begin{align*}  
S(f,g)&= \frac{\mathrm{lcm}(\ini_{<_\text{lex}}(f),\ini_{<_\text{lex}}(g))}{\ini_{<_\text{lex}}(f)}f- \frac{\mathrm{lcm}(\ini_{<_\text{lex}}(f),\ini_{<_\text{lex}}(g))}{\ini_{<_\text{lex}}(g)}g \\
& = z_{q}f-z_{j}g \\
& = z_{q}(x_{i}y_{i}z_{j}-z_{i}x_{j}y_{j}) - z_{j}(x_{i}y_{i}z_{q}-z_{i}x_{q}y_{q})  \\
& = z_{i}(x_{q}y_{q}z_{j}-z_{q}x_{j}y_{j}).
\end{align*}
Note that, up to sign, $x_{q}y_{q}z_{j}-z_{q}x_{j}y_{j}$ is a generator of $I_{\calG_{n}}$ and therefore $S(f,g)$ will reduce to $0$. The $j=q$ case is similar.
\end{proof}

\begin{corollary}
 The initial ideal $in_{<_\text{lex}}\big( I_{\mathcal{G}_{n}} \big)$ of $I_{\mathcal{G}_{n}}$ with respect to the monomial order $<_\text{lex}$ 
is generated by the squarefree monomials
\begin{align}\label{eq:initial}
x_{i}y_{i}z_{j};\hspace{0.5cm} 1\leq i < j\leq n.  
\end{align}
\end{corollary}
By this corollary, since the given monomial ideal is squarefree, 
we can associate a simplicial complex whose Stanley-Reisner ideal coincides with the initial ideal generated by \eqref{eq:initial}. 


\section{Computation of the $h$-polynomial of $\kk[\calG_n]$}\label{sec:h-poly}

Let $\Delta_n$ be the simplicial complex whose Stanley-Reisner ideal coincides with the initial ideal of the toric ideal $I_{\mathcal{G}_{n}}$ with respect to $<_\text{lex}$. 
(For the introduction to Stanley-Reisner theory, consult, e.g., \cite[Section 5]{BH}.) 
%
%
Let $\mathcal{F}(\Delta_{n})$ be the set of all facets of $\Delta_{n}$. 
By definition, any facet of our simplicial complex $\Delta_{n}$ can be expressed 
as the maximal set that does not contain the triplet $\{x_{i},y_{i},z_{j}\};\hspace{0.2cm} 1\leq i < j\leq n.$ 
Since $x_{n} ,y_{n}$ and $z_{1}$ will be contained in all the facets, with out loss of generality, 
we write the facets without indicating these elements. Therefore, any $F\in \mathcal{F}(\Delta_{n})$ can be expressed as: 
$$F= \bigcup\limits_{\substack{i\in I \\1\leq i \leq n-1}}\{x_{i}\}\cup \bigcup\limits_{\substack{j\in J \\1\leq j \leq n-1}}\{y_{j}\}\cup 
\bigcup\limits_{\substack{k\in K \\2\leq k \leq n}}\{z_{k}\},$$
which is maximal and does not contain the triplet $\{x_{i},y_{i},z_{j}\};\hspace{0.2cm} 1\leq i < j\leq n.$ 
Let us try to get a more concrete representation for the facets in $\mathcal{F}(\Delta_{n})$.
Consider any $F \in \mathcal{F}(\Delta_{n})$. 

\hspace{-0.5cm}$\textbf{\underline{Case 1:}}$ 

Let $i\in I\cap J$. This implies $z_{i+1},\dots ,z_{n}\notin F$, since $F$ does not contain the triplet $\{x_{i},y_{i},z_{j}\};\hspace{0.2cm} 1\leq i < j\leq n. $

\hspace{-0.5cm}$\textbf{\underline{Case 2:}}$ 

Let us consider $i\in I$. If there exists some $k$ with $i< k\leq n$ such that $z_{k}\in F$, then we have $y_{i}\notin F,$ that is, $i\in I\backslash J.$

If $z_{k}\notin F$ for all $k$ with $i< k\leq n$, then by the maximality of $F$, we have $y_{i}\in F$ and $i\in I\cap J$.

\hspace{-0.5cm}$\textbf{\underline{Case 3:}}$ 

Let us consider $j\in J$. If there exists some $k$ with $j< k\leq n$ such that $z_{k}\in F$, then we have $x_{j}\notin F,$ that is, $j\in J\backslash I.$

If $z_{k}\notin F$ for all $k$ with $j < k\leq n$, then by the maximality of the set $F$, we have $x_{j}\in F$ and $j\in I\cap J$. 

\medskip

According to our observations from the three situations above, any facet in $\mathcal{F}(\Delta_{n})$ is of the form:
\begin{align}\label{eq:facets}\{w_{1},\dots ,w_{j-1}\}\cup \{x_{j},y_{j},\dots , x_{n-1},y_{n-1}\}\cup \{z_{2},\dots , z_{j}\},\end{align}
where $w_{i}\in\{x_{i},y_{i}\}$ and $j=1,\dots , n.$

\begin{remark}[Mac-Mullen characterization of $h$-vectors]
Here, we recall a well-known method to compute the $h$-vector of a shellable simplicial complex. 
(For example, see \cite[Corollary 5.1.14]{BH}.) 

Let $F_{1},\dots ,F_{t}$ be the facets of a pure simplicial complex $\Delta$. 
Let $\langle F_1,\ldots,F_m \rangle$ be the unique smallest simplicial complex which contains all $F_i$, $1 \leq i \leq m$. 
The ordering of the facets is said to be a \textit{shelling} if it satisfies that 
$\langle F_i \rangle \cap \langle F_1,\ldots,F_{i-1} \rangle$ is generated by a non-empty set of maximal proper faces of $F_i$ for all $2 \leq i \leq t$. 
We say that a pure simplicial complex is \textit{shellable} if it has a shelling. 
Throughout our further study, we may refer the subcomplex $\langle F_i \rangle \cap \langle F_1,\ldots,F_{i-1} \rangle$ 
as the \textit{intersection subcomplex} corresponding to the $i^{\textrm{th}}$ shelling step. 
Let $r_i$ be the number of maximal proper faces of $F_i$ in $\langle F_i \rangle \cap \langle F_1,\ldots,F_{i-1} \rangle$ for $2 \leq i \leq t$, and let $r_1=0$. 
Then the $h$-vector of $\Delta$, $h(\Delta) = (h_0,\dots ,h_d)$, is obtained by $ h_i  =  |\{j \colon r_j = i\}| .$
\end{remark}

Now, we consider an ordering of the facets $F_1^1,\ldots,F_{t_n}^n$ of $\Delta_n$. 
From the structure of each facet as shown in \eqref{eq:facets}, we have $t_n=\sum\limits_{i=0}^{n-1}2^i$ and $|F_i^n|=2n-2$ for all $1 \leq i \leq t_n$. 
Let us consdier each facet as a $(2n-2)$-tuple of $x_1,y_1,x_2,y_2,z_2,\ldots,x_{n-1}, y_{n-1}$, $z_{n-1},z_n$. Lexicographic order $<_L$ in $\mathcal{F}(\Delta_n)$ is defined by 
$$(a_1,\ldots,a_{2n-2}) <_L (b_1,\ldots,b_{2n-2})$$ 
if and only if either $a_i \neq b_i$ for some $i$ and $a_i<_\text{lex}b_i$ with respect to \eqref{eq:lex}. 
Now, we consider an ordering of the facets $F_1^n,\ldots,F_{t_n}^n$ of $\Delta_n$ 
such that they are arranged in lexicographically increasing order of their corresponding $(2n-2)$-tuple.

\medskip

Let $r^{n}_{i}$ be the number of maximal proper faces of $F^{n}_{i}$ that generates the $i^\mathrm{th}$ intersection subcomplex for $2 \leq i \leq t_n$. 
We define $\delta_{n}=\{r^{n}_2,\dots ,r^{n}_{t_{n}}\}$ with $n\geq 2$ as a multi-set. 
\begin{lemma}\label{lem:induction}
For each $n\geq 2,$ we have 
$$\delta_{n+1}=\{1,\delta_{n},2,\delta_{n}+1\},$$
where $\delta_n+1=\{\alpha+1 : \alpha \in \delta_n\}$. 
\end{lemma}

By induction on $n$, and using Lemma~\ref{lem:induction}, we can obtain the $h$-vector of $\Delta_n$. 
Our goal is to show that $$h(\Delta_n)=\Bigg{(}\binom{n}{0},\binom{n}{1}-1,\binom{n}{2},\binom{n}{3},\dots ,\binom{n}{n}\Bigg{)}\;\; \text{ for any} \;\; n\geq 2.$$ 

For $n=2$, since $\calF(\Delta_2)=\{\{x_1,y_1\},\{x_1,z_2\},\{y_1,z_2\}\}$, 
we obtain the $h$-vector as $(1,1,1)$ which is equal to our formula $\Big{(}\binom{2}{0},\binom{2}{1}-1,\binom{2}{2}\Big{)}$.

By the hypothesis of induction, assume that our formula holds for an arbitrary $n$. 
Therefore, the $h$-vector associated with $\Delta_{n}$ is $\Big{(}\binom{n}{0},\binom{n}{1}-1,\binom{n}{2},\binom{n}{3},\dots ,\binom{n}{n}\Big{)}$.
Let $h(\Delta_{n+1})=(h_0^{n+1},h_1^{n+1},\ldots,h_{n+1}^{n+1})$. By Lemma~\ref{lem:induction}, we see the following: 
\begin{align*}
&h^{n+1}_{0}= 1=\binom{n+1}{0}, \\
&h^{n+1}_{1}= 1+h^{n}_{1}=1+\binom{n}{1}-1=\binom{n+1}{1}-1, \\
&h^{n+1}_{2}= 1+h^{n}_{2}+h^{n}_{1}=1+\binom{n}{2}+\binom{n}{1}-1=\binom{n+1}{2}, \\
&h^{n+1}_{i}= h^{n}_{i}+h^{n}_{i-1}=\binom{n}{i}+\binom{n}{i-1}=\binom{n+1}{i} \;\;\text{for}\;\; 2\leq i\leq n.
\end{align*}

Hence by induction, the $h$-vector associated with simplicial complex $\Delta_{n}$ is 
$$\Bigg{(}\binom{n}{0},\binom{n}{1}-1,\binom{n}{2},\binom{n}{3},\dots ,\binom{n}{n}\Bigg{)}\;\;\text{for any}\;\; n\geq 2.$$

Since the Hilbert series of $S/I$ coincides with that of $S/\ini_{<_\text{lex}}(I)$ in general (see, e.g., \cite[Proposition 2.6]{HHO}), 
we conclude the above as the desired $h$-vector of $\kk[\calG_n]$.

\bigskip

The remaining section is devoted to illustrating that our ordering of the facets gives a shelling of $\Delta_n$ and establishing Lemma~\ref{lem:induction}. 

For each $n\geq 2,$ 
we observe that $r^{n}_{2}=1$, $t_{n}=|\mathcal{F}(\Delta_{n})|=\sum\limits_{i=0}^{n-1}2^{i}$ and we have $t_{n+1}=1+2t_{n}$. 

As per our ordering, the facets are ordered in such a manner that 
the facets consisting of $x_{1}$ comes first and after the $\big(\frac{t_{n}+1}{2}\big)^{\textrm{th}}$ stage, 
the pattern of facet ordering repeats in the exact same manner as that of $F^{n}_{2},\dots ,F^{n}_{\frac{t_{n}+1}{2}}$, 
and consists of same elements except for $x_{1}$ replaced with $y_{1}$. 
More precisely, we have $F_{\frac{t_{n}+1}{2}+i}=(F_{i+1} \setminus \{x_1\}) \cup \{y_1\}$ for $1\leq i\leq \frac{t_{n}-1}{2}$. 
Therefore, the corresponding intersection subcomplex for each $k^{\textrm{th}}$ shelling step, $\frac{t_{n}+3}{2} \leq k \leq t_{n} $, 
will definitely contain the maximal face $$ F^{n}_{i+1} \setminus \{x_{1}\}\hspace{0.2cm}\text{for any }1\leq i\leq \frac{t_{n}-1}{2}.$$
Moreover, when we look at the first $\frac{t_n+1}{2}$ facets, we observe that each of the facets differs from the preceding one by just an element. Thus with these observations, we claim that $F_1^n,\ldots,F_{t_n}^n$ is a shelling of $\Delta_n$.
    
Since the ordering pattern repeats after the $\big(\frac{t_{n}+1}{2}\big)^{\textrm{th}}$ facet, for $1 \leq i \leq \frac{t_n-1}{2}$, 
we see that $r_{i+1}^n$ maximal proper faces will always be contained in the intersection subcomplex for each $k^\textrm{th}$ shelling step, $\frac{t_n+1}{2} +i \leq k \leq t_n$
and the intersection subcomplex also contains the maximal face $F_{i+1}^n \setminus \{x_1\}$. 
Therefore, we have $$r^{n}_{\frac{t_{n}+1}{2}+i}=1+r^{n}_{i+1}\hspace{0.2cm}\text{for any }1\leq i\leq \frac{t_{n}-1}{2}.$$ 

Hence, we can express $\delta_{n}$ as $$ \delta_{n}=\{1,r^{n}_{3},\dots,r^{n}_{\frac{t_{n}+1}{2}},2,r^{n}_{3}+1,\dots ,r^{n}_{\frac{t_{n}+1}{2}}+1\}.$$

The set $\mathcal{F}(\Delta_{n})$ consists of \eqref{eq:facets}. Namely, each facet $F^{n}_{k}$, $1\leq k\leq t_{n}$, can be denoted as:
$$ \bigcup\limits_{i=1}^{j-1}\{w_{i}\}\cup \bigcup\limits_{i=j}^{n-1}\{x_{i},y_{i}\}\cup\bigcup\limits_{i=2}^{j} \{z_{i}\},\hspace{0.4cm} j=1,\dots , n.$$
   
For each $1\leq j \leq n$, we have $2^{j-1}$ number of facets corresponding to it. 
   We can show that there exists a one-to-one correspondence between the facets $F^{n}_{k}$ and $F^{n+1}_{k+1}$ for all $n\geq 2$ and $ 2\leq k\leq t_{n}$. 
   The one-to-one correspondence is given by
   $$\phi \colon \Big\{F_{k}^{n}\colon 2\leq k\leq t_{n}\Big\}\longrightarrow \Big\{F_{k'}^{n+1}\colon 3\leq k'\leq \frac{t_{n+1}+1}{2}\Big\} $$
   $$ \phi (F^{n}_{k}) = F^{n+1}_{k+1},\hspace{0.5cm} 2\leq k\leq t_{n},$$
   $$  \bigcup\limits_{i=1}^{j-1}\{w_{i}\}\cup \bigcup\limits_{i=j}^{n-1}\{x_{i},y_{i}\}\cup\bigcup\limits_{i=2}^{j} \{z_{i}\} 
   \longmapsto \{x_{1}\}\cup \bigcup\limits_{i=1}^{j-1}\{w_{i+1}\}\cup \bigcup\limits_{i=j}^{n-1}\{x_{i+1},y_{i+1}\}\cup\bigcup\limits_{i=2}^{j+1} \{z_{i}\},$$
   with $j=1,\dots , n.$
   
   In order to prove this one-to-one correspondence, we show that any facet in the set, 
   $\Big\{F_{k'}^{n+1}\colon 3\leq k'\leq \frac{t_{n+1}+1}{2}\Big\}$ will be of the form 
   $$\{x_{1}\}\cup \bigcup\limits_{i=1}^{j-1}\{w_{i+1}\}\cup \bigcup\limits_{i=j}^{n-1}\{x_{i+1},y_{i+1}\}\cup\bigcup\limits_{i=2}^{j+1} \{z_{i}\},$$
   corresponding to the facet $\bigcup\limits_{i=1}^{j-1}\{w_{i}\}\cup \bigcup\limits_{i=j}^{n-1}\{x_{i},y_{i}\}\cup\bigcup\limits_{i=2}^{j} \{z_{i}\} \; (2\leq j \leq n)$
   in the set $\Big\{F_{k}^{n}\colon 2\leq k\leq t_{n}\Big\}$. 
   
   According to our shelling, any facet of the simplicial complex $\Delta_{n+1}$ containing the element $\{x_{1}\}$ 
   belongs to $\Big\{F_{k}^{n+1}\colon 1\leq k\leq \frac{t_{n+1}+1}{2}\Big\}$ with $1\leq j\leq n+1$. 
   So, when we consider some $F^{n+1}_{k'}\in \Big\{F_{k}^{n+1}\colon 3\leq k\leq \frac{t_{n+1}+1}{2}\Big\}$, it can be expressed as:
   $$ F^{n+1}_{k'} = \{x_{1}\}\cup \bigcup\limits_{i=2}^{j'-1}\{w_{i}\}\cup \bigcup\limits_{i=j'}^{n}\{x_{i},y_{i}\}\cup\bigcup\limits_{i=2}^{j'} \{z_{i}\},\hspace{0.3cm} 3\leq j'\leq n+1.$$
   We can write it as
   \begin{align*}
   F^{n+1}_{k'}& = \{x_{1}\}\cup \bigcup\limits_{i=2}^{j'-1}\{w_{i}\}\cup \bigcup\limits_{i=j'}^{n}\{x_{i},y_{i}\}\cup\bigcup\limits_{i=2}^{j'} \{z_{i}\} \\
   & = \{x_{1}\}\cup \bigcup\limits_{i=1}^{j'-2}\{w_{i+1}\}\cup \bigcup\limits_{i=j'-1}^{n-1}\{x_{i+1},y_{i+1}\}\cup\bigcup\limits_{i=2}^{j'} \{z_{i}\} \;\;\text{for}\;\; 3\leq j'\leq n+1.
   \end{align*}
   By expressing it in terms of $2\leq j \leq n$, we have
   $$F^{n+1}_{k'}=\{x_{1}\}\cup \bigcup\limits_{i=1}^{j-1}\{w_{i+1}\}\cup \bigcup\limits_{i=j}^{n-1}\{x_{i+1},y_{i+1}\}\cup\bigcup\limits_{i=2}^{j+1} \{z_{i}\}
   \;\text{ for }\;2\leq j\leq n.$$

Hence, we have the one-to-one correspondence. This one-to-one correspondence guarantees that $r^{n+1}_{i+1}=r^{n}_{i}$ for $2\leq i\leq t_{n}.$

We also observe that 
   $$t_{n+1}=1+2t_{n} \implies t_{n}=\frac{t_{n+1}-1}{2} \implies t_{n}+1=\frac{t_{n+1}+1}{2}.$$
From all the above observations, we have 
   \begin{align*}
   \delta_{n+1}&=\{1,r^{n+1}_{3},\dots ,r^{n+1}_{t_{n}+1},2,r^{n+1}_{3}+1,r^{n+1}_{t_{n}+1}+1\} \\
   \implies\delta_{n+1}&=\{1,r^{n}_{2},\dots ,r^{n}_{t_{n}},2,r^{n}_{2}+1,r^{n}_{t_{n}}+1\} \\
   \implies\delta_{n+1}&=\{1,\delta_{n},2,\delta_{n}+1\}.
   \end{align*}

\bigskip

\section{On the almost Gorensteinness of $\kk[\calG_n]$}\label{sec:almGor}

In this section, we are in the process of obtaining a theoretical proof to state that $\kk[\mathcal{G}_{n}]$ is almost Gorenstein for all $n\geq 2$. 

We first recall the necessary and sufficient condition for a homogeneous ring to be almost Gorenstein from \cite{H}. 
\begin{prop}[{\cite[Corollary 2.7]{H}}]\label{prop:alm}
Let $R$ be a Cohen--Macaulay homogeneous ring of dimension $d$ over a field $\kk$ and let $h(R)=(h_0,h_1,\ldots,h_s)$ be its $h$-vector. 
Then $R$ is almost Gorenstein if and only if the following equality holds: 
\begin{align*}
r(R)-1=\sum_{j=0}^{s-1}((h_s+\cdots+h_{s-j})-(h_0+\cdots+h_j)) =:\tilde{e}(R). 
\end{align*}
Note that $r(R)-1 \leq \tilde{e}(R)$ is always satisfied, so almost Gorensteinness is equivalent to the inequality $r(R) \geq \tilde{e}(R) + 1$. 
\end{prop}
Let us assume that $R$ is a domain. Then there is an ideal $I_R$ which is isomorphic to a canonical module of $R$ as an $R$-module. 
We know that the number of elements of a minimal system of generators of $I_R$ coincides with $r(R)$. 

\bigskip

Let us consider $R=\kk[\calG_n]$. Note that $\calG_n$ satisfies the odd cycle condition, so $\kk[\calG_n]$ is normal. 
By the first part of Theorem~\ref{thm:main}, we can compute $\tilde{e}(R)$ as follows: 
\begin{align*}
&\sum_{j=0}^{n-1}\left\{\left(\binom{n}{n}+\cdots+\binom{n}{n-j}\right) - \left(\binom{n}{0}+\big(\binom{n}{1}-1\big)+\cdots+\binom{n}{j}\right)\right\}\\
=&\sum_{j=1}^{n-2}1=n-2. 
\end{align*}
Hence by Proposition~\ref{prop:alm}, it is enough to show that $r(R) \geq n-1$. 
Furthermore, $r(R)$ is equal to the number of elements of a minimal system of generators of $I_R$, 
which is the relative interior of $\QQ_{\geq 0}A_{\calG_n} \cap \ZZ A_{\calG_n}$ (see \cite[Theorem 6.3.5]{BH}). 
Let $C=\QQ_{\geq 0}A_{\calG_n} \subset \RR^{2n+1}$. 
In what follows, it suffices to show that we need at least $(n-1)$ elements as a minimal system of generators for the relative interior of the cone $C$. 

Let us denote the vertices of $\calG_n$ as follows: 
\begin{align*}
&V(\calG_n)=\{u_i^{(1)},u_i^{(2)} : i=1,\ldots,n\} \cup \{w\} \;\;\text{ and we let } \\
&x_i=\{w,u_i^{(1)}\}, \; y_i=\{w,u_i^{(2)}\} \text{ and }z_i=\{u_i^{(1)},u_i^{(2)}\} \text{ for } i=1,\ldots,n. 
\end{align*}
We use the following notation for each entry of $\RR^{2n+1}$: 
\begin{align*}
\RR^{2n+1}=\{c_{1,1}\eb_{1,1}+\cdots+c_{1,n}\eb_{1,n}+c_{2,1}\eb_{2,1}+\cdots+c_{2,n}\eb_{2,n}+c'\eb' : c_{1,i},c_{2,i}, c' \in \RR\}, 
\end{align*}
where $\eb_{1,i},\eb_{2,i},\eb'$ are the unit vectors of $\RR^{2n+1}$, 
each $\eb_{1,i}$ (resp. $\eb_{2,i}$) corresponds to $u_i^{(1)}$ (resp. $u_i^{(2)}$) and $\eb'$ corresponds to $w$. 


For $j=1,\ldots,n-1$, let $$\alpha_j:=\sum_{i=1}^n(\eb_{1,i}+\eb_{2,i})+2j\eb'.$$
In what follows, we verify that $\alpha_j \in C^\circ \cap \ZZ^{2n+1}$, where $C^\circ$ denotes the relative interior of $C$,  
and they should be included in a minimal system of generators of $I_R$. 

\noindent
{\bf The first step}: We check that $\alpha_j \in C^\circ$. Here, we see the following: 
\begin{itemize}
\item Each of $u_i^{(1)}$s and $u_i^{(2)}$s is a regular vertex of $\calG_n$, while $w$ is not. 
\item A subset $T$ of $V(\calG_n)$ is fundamental if and only if $T=\{w\}$ or $T=\{u_1,\ldots,u_n\}$, where $u_i \in \{u_i^{(1)},u_i^{(2)}\}$ for each $i$. 
\end{itemize}
Hence, it follows from \eqref{eq:ineq} that 
$\sum_{i=1}^n(c_{1,i}\eb_{1,i}+c_{2,i}\eb_{2,i})+c'\eb' \in \RR^{2n+1}${\color{blue},} belongs to $C$ if and only if the following inequalities are satisfied: 
\begin{equation}\label{eq:ineq_facets}
\begin{split}
&c_{1,i} \geq 0 \;\;\text{and}\;\; c_{2,i} \geq 0 \;\;\text{ for any }i=1,\ldots,n, \\
&\sum_{i=1}^n(c_{1,i}+c_{2,i}) \geq c', \\
&\sum_{i \in U}c_{i,1}+\sum_{i \in [n] \setminus U}c_{i,2}+c' \geq \sum_{i \in [n] \setminus U}c_{i,1}+\sum_{i \in U}c_{i,2} \;\;\text{ for any }U \subset [n].
\end{split}
\end{equation}
It is straightforward to check that $\alpha_j$ satisfies these inequalities with strict inequalities for each $j$. This implies that $\alpha_j \in C^\circ$. 

\noindent
{\bf The second step}: We prove that $\alpha_j$ cannot be written 
as a sum of an element in $C^\circ \cap \ZZ^{2n+1}$ and an element in $C \cap \ZZ^{2n+1} \setminus \{{\bf 0}\}$. 

Suppose that $\alpha_j=\alpha' + \beta$ for some $\alpha' \in C^\circ \cap \ZZ^{2n+1}$ and $\beta \in C \cap \ZZ^{2n+1} \setminus \{{\bf 0}\}$. 
Let $$\alpha'=\sum_{i=1}^na_{1,i}'\eb_{1,i}+\sum_{i=1}^na_{2,i}'\eb_{2,i}+a'\eb' \;\text{ and }\;\beta=\sum_{i=1}^nb_{1,i}\eb_{1,i}+\sum_{i=1}^nb_{2,i}\eb_{2,i}+b\eb'.$$ 
Then we see that $a_{1,i}' \geq 1$ and $a_{2,i}' \geq 1$ for $1 \leq i \leq n$ (see \eqref{eq:ineq_facets}). Hence, $b_{1,i} \leq 0$ and $b_{2,i} \leq 0$. 
On the other hand, $b_{1,i} \geq 0$ and $b_{2,i} \geq 0$ should be also satisfied, so we obtain that $\beta=b\eb'$. 
Since $\beta \neq {\bf 0}$, by the second inequality of \eqref{eq:ineq_facets}, we have $b<0$, a contradiction to the third inequality.

\noindent
{\bf The third step}: By the first and second steps, we see that $\alpha_1,\ldots,\alpha_{n-1}$ are required for a minimal system of generators of $I_R$.  
This shows that $r(R) \geq n-1$, as required. 

\bigskip

\section*{Acknowledgements}
The first named author is partially supported by JSPS Grant-in-Aid for Scientists Research (B) 18H01134 and Scientists Research (C) 20K03513.

\end{document}